\documentclass[a4paper,12pt]{article}
\usepackage{a4wide}
\usepackage{amsmath}
\usepackage{amssymb}
\usepackage{amsthm}
\usepackage{latexsym}
\usepackage{graphicx}
\usepackage[english]{babel}
\usepackage{makeidx}
\usepackage{mathtools}
              
\newtheorem{dfn} [subsection]{Definition}
\newtheorem{obs} [subsection]{Remark}

\newtheorem{prop}[subsection]{Proposition}

\newtheorem{teor}[subsection]{Theorem}

\newtheorem{cor} [subsection]{Corollary}

\newcommand{\OGR}{\mathcal F(\Gamma,R)}
\newcommand{\OGM}{\mathcal F(\Gamma,M)}
\newcommand{\OGN}{\mathcal F(\Gamma,N)}
\newcommand{\OGU}{\mathcal F(\Gamma,U)}
\newcommand{\wF}{\mathcal F^f}

\def\Ree{\operatorname{Re}}
\def\Der{\operatorname{Der}}
\def\Ker{\operatorname{Ker}}

\def\RMod{\operatorname{R-Mod}}

\numberwithin{equation}{section}

\usepackage{titletoc}
\dottedcontents{section}[2.5em]{}{2em}{1pc}

\begin{document}
\selectlanguage{english}
\frenchspacing

\large
\begin{center}
\textbf{On the algebraic properties of the ring of Dirichlet convolutions}

Mircea Cimpoea\c s
\end{center}
\normalsize

\begin{abstract}
Let $R$ be a commutative ring and $\Gamma$ a commutative monoid of finite type. We study algebraic properties of
modules and derivations over the associated ring $\mathcal F(\Gamma,R)$ of Dirichlet convolutions.
If $\Gamma$ is cancellative and $G(\Gamma)$ is its associated Grothendieck group, we construct a
natural extension $\mathcal F^f(G(\Gamma),R)$ of $\mathcal F(\Gamma,R)$ and we study its basic properties. Further properties are discussed in the case $\Gamma=\mathbb N^*$ and $G(\Gamma)=\mathbb Q^*_+$. In particular, we show that 
$\mathcal F^f(\mathbb Q^*_+,R)\cong R[[x_1,x_2,\ldots,]][x_1^{-1},x_2^{-1},\ldots]$.


\end{abstract}


\section*{Introduction}

Let $(\Gamma,\cdot)$ be a commutative monoid of finite type and let $R$ be a commutative ring with unity. 
In the set of functions defined on $\Gamma$ with values in $R$,
$\OGR:=\{\alpha:\Gamma\rightarrow R\}$,
we consider the operations
$$ (\alpha+\beta)(n):=\alpha(n)+\beta(n),\;(\forall)n\in \Gamma \text{ and } 
 (\alpha \cdot \beta)(n):=\sum_{ab=n}\alpha(a)\beta(b),\;(\forall)n\in \Gamma.$$
It is well known that $(\OGR,+,\cdot)$ is a commutative ring with unity, see for instance \cite{ber}.
Let $M$ be a $R$-module. Let $\mathcal F(\Gamma,M):=\{f:\Gamma \rightarrow M\}$.
For any $f,g\in \OGM$ and $\alpha\in \OGR$, we define:
$$ (f+g)(n):=f(n)+g(n),\;(\forall)n\in \Gamma\;\text{ and }
 (\alpha \cdot f)(n):=\sum_{ab=n}\alpha(a)f(b),\;(\forall)n\in \Gamma.$$
We show that $\mathcal F(\Gamma,M)$ has a natural structure of a $\mathcal F(\Gamma,R)$-module
and we discuss certain properties of the functor $M\mapsto \mathcal(\Gamma,M)$, see Proposition $1.3$. Given a morphism of monoids
$L:(\Gamma,\cdot)\rightarrow (M,+)$, we show that the induced map 
$\Phi_{L,M}:\mathcal F(\Gamma,M) \rightarrow \mathcal F_{L}(\Gamma,M),\;\Phi_{L,M}(f)(n):=L(n)f(n)$
is a morphism of $\mathcal F(\Gamma,R)$-modules, see Proposition $1.6$.

Let $A$ be a commutative ring with unity and let $i:A\rightarrow R$ be a morphism of rings with unity. 
Let $M$ be an $R$-module. An $A$-derivation $D:R \rightarrow M$ is an $A$-linear map, satisfying the Leibniz rule, i.e. 
$D(fg)=fD(g)+gD(f), (\forall)f,g\in R$.
The set of $A$-derivations $\Der_A(R,M)$ has a natural structure of a $R$-module. 
Let $D\in \Der_A(R,M)$ and let $\delta:\Gamma \rightarrow (M,+)$ be a morphism of monoids. In Proposition $1.7$, we prove that
$$\widetilde D:\mathcal F(\Gamma,R) \rightarrow \mathcal F(\Gamma,M),\; \widetilde D(\alpha)(n)=D(\alpha(n))+\alpha(n)\delta(n),\;(\forall)n\in\Gamma, \text{ is an $A$-derivation}.$$ 
Assume that the monoid $(\Gamma,\cdot)$ is cancellative, i.e. $xy = xz$ implies $y=z$. Let $G(\Gamma)$ be the Grothendieck group 
associated to $\Gamma$, see \cite{atiyah}. In the set
$$\wF(G(\Gamma),R):=\{\alpha:G(\Gamma) \rightarrow R\;:\; (\exists)d\in\Gamma \text{ such that } (\forall)q\in G(\Gamma), 
\alpha(q)\neq 0 \Rightarrow dq\in\Gamma\}.$$
we consider the operations
$$ (\alpha+\beta)(q):=\alpha(q)+\beta(q),\;(\forall)q\in G(\Gamma) \text{ and } 
 (\alpha \cdot \beta)(q):=\sum_{q'q''=q}\alpha(q')\beta(q''),\;(\forall)q\in G(\Gamma).$$
We prove that $\wF(G(\Gamma),R)$ is an extension of the ring $\mathcal F(\Gamma,R)$, see Proposition $3.2$.
Let $M$ be an $R$-module. We consider the set
$$\wF(G(\Gamma),M):=\{f:G(\Gamma) \rightarrow M\;:\; (\exists)d\in\Gamma \text{ such that } (\forall)q\in G(\Gamma) 
f(q)\neq 0 \Rightarrow dq\in\Gamma\}.$$
For $f,q\in\wF(G(\Gamma),M)$ we define $(f+g)(q):=f(q)+g(q),(\forall)q\in G(M)$.
Given $\alpha\in \wF(G(\Gamma),R)$ and $f\in\wF(G(\Gamma),M)$ we define 
$(\alpha\cdot f)(q):= \sum_{q'q''=q}\alpha(q')f(q''),\;(\forall)q\in G(\Gamma)$.
We prove that $\wF(G(\Gamma),M)$ has a structure of an $\wF(G(\Gamma),R)$-module and we study the
connections between the associations $M\mapsto \mathcal F(\Gamma,M)$ and $M\mapsto \wF(\Gamma,M)$, see Proposition $3.3$.
In particular, in Proposition $3.4$ and Remark $3.5$, we show that if $D\in \Der_A(R,M)$ is an $A$-derivation and $\delta:\Gamma \rightarrow (M,+)$ is a morphism of monoids,
then we can construct an $A$-derivation on $\overline D:\wF(G(\Gamma),R)\rightarrow \wF(G(\Gamma),M)$ which
extend $\widetilde D$.

The most important case, largely studied in analytic number theory \cite{apostol}, is the case when $R$ is a domain (or even more
particullary, when $R=\mathbb C$) and 
$\Gamma=\mathbb N^*$ is the multiplicative monoid of positive integers. Cashwell and Everett showed in \cite{cash} that 
$\mathcal F(\mathbb N^*,R)$ is also a domain. Moreover, if $R$ is an UFD with the property that $R[[x_1,\ldots,x_n]]$ are
UFD for any $n\geq 1$, then $\mathcal F(\mathbb N^*,R)$ is also an UFD, see \cite{cash2}. It is well known that the Grothendieck
group associated to $\mathbb N^*$ is $\mathbb Q_{+}^*:=$ the group of positive rational numbers. We show that
$$\wF(\mathbb Q^*_+,R) \cong R[[x_1,x_2,\ldots]][x_1^{-1},x_2^{-1},\ldots],$$
and, in particular, if $R[[x_1,x_2,\ldots]]$ is UFD, then $\wF(\mathbb Q^*_+,R)$ is an UFD, see Proposition $3.6$ and Corollary $3.7$.
We also note in the Remark $3.8$ the connections with the field of Hahn series \cite{hahn} associated to a ordered group.

Let $U\subset \mathbb C$ be an open set and let $\mathcal O(U)$ be the ring of holomorphic functions defined in $U$ with values in $\mathbb C$. It is well known that $\mathcal O(U)$ is a domain. We consider 
$$\widetilde D: \mathcal F(\mathbb N^*,\mathcal O(U)) \rightarrow \mathcal F(\mathbb N^*,\mathcal O(U)), \; \widetilde  D(\alpha)(n)(z):=\alpha(n)'(z) - \alpha(n)(z) \log n, \; (\forall)n\in \mathbb N, z\in U.$$
We note that $\widetilde D$ is a $\mathbb C$-derivation on $\mathcal F(\mathbb N^*,\mathcal O(U))$, see Proposition $2.9$.
Assume that the series of functions $F_{\alpha}(z):=\sum_{n=1}^{\infty} \frac{\alpha(n)(z)}{n^z}$, $z\in U$, and
and $G_{\alpha}(z)= \sum_{n=1}^{\infty} \left(\frac{\alpha(n)(z)}{n^z}\right)'$, $z\in U$, are uniformly convergent on the compact subsets $K\subset U$. It is well known that, in this case, $F$ defines a derivable
(holomorphic) function on $U$ and, moreover, $F'=G$. It is easy to see that $F'_{\alpha} = F_{\widetilde D(\alpha)}$. See also Remark $2.10$.

Further connections between the $\mathbb C$-linear independence of 
$\alpha_1,\ldots,\alpha_m\in \mathcal F(\mathbb N^*,\mathcal O(U))$, with some suplimentary conditions,
 and the linear independence of the series $F_{\alpha_1},\ldots,F_{\alpha_m}$ and their derivatives, over the field of meromorphic function of order $<1$, were made in \cite{cimf} and \cite{cim}, but they
are beyond the scope of our paper.

\newpage
\section{The ring of Dirichlet convolutions of a monoid of finite type}

A \emph{semigroup} (written multiplicatively) is a nonempty set $\Gamma$ with a binary operation $x\cdot y$ that is associative. 
A \emph{monoid} is a semigroup having an element $1$ that is neutral for the operation.
A monoid $(\Gamma,\cdot,1)$ is said to be of \emph{finite type} if 
for each $n\in \Gamma$, the set of ordered pairs $\{(a, b) \in \Gamma \times \Gamma \; : \; n = ab\}$ is finite. 
Let $R$ be a ring with unity. 

In the set of functions defined of $\Gamma$ with values in $R$,
$$\OGR:=\{\alpha:\Gamma\rightarrow R\},$$
we consider the operations
$$ (\alpha+\beta)(n):=\alpha(n)+\beta(n),\;(\forall)n\in \Gamma, $$
$$ (\alpha \cdot \beta)(n):=\sum_{ab=n}\alpha(a)\beta(b),\;(\forall)n\in \Gamma.$$
We denote $0,e \in \OGR$ the functions 
$$0(n)=0,(\forall)n\in\Gamma, \; e(1)=1, e(n)=0, \; (\forall)n\in \Gamma\setminus \{1\}.$$
Given $\alpha\in \OGR$, $-\alpha$ is the function defined by 
$$(-\alpha)(n):=-\alpha(n), \; (\forall)n\in \Gamma.$$
It is easy to check that $(\OGR,+,0)$ is an abelian group and $-\alpha+\alpha=0$, $(\forall)\alpha\in \OGR$.
Also, by straigtforward computations, we can show that $(\OGR,\cdot,e)$ is a monoid and $\cdot$ is distributive with respect to $+$. 
On the other hand, if $\Gamma$ and $R$ are commutative, then 
$$(\alpha\cdot \beta)(n) = \sum_{ab=n}\alpha(a)\beta(b) = \sum_{ba=n}\alpha(b)\beta(a) = (\beta \cdot \alpha)(n),\;(\forall)n\in \Gamma.$$
Hence, we have the following result.

\begin{prop}(\cite[Theorem 5]{ber})
$(\OGR,+,\cdot)$ is a ring with unity. Moreover, if $\Gamma$ and $R$ are commutative, then $\OGR$ is commutative.
\end{prop}

In the following, we will assume that $R$ is a commutative ring with unity and $\Gamma$ is a commutative monoid of finite type.

\begin{obs}\emph{
 The map $i:R \rightarrow \OGR$ defined by $i(r)(1)=r$ and $i(r)(n)=0$, $(\forall)n\geq 2$, is a ring monomorphism, hence $\OGR$ has a natural structure of a $R$-algebra.
 }
\end{obs}

Let $M$ be a $R$-module. We denote $\mathcal F(\Gamma,M):=\{f:\Gamma \rightarrow M\}$,
the set of functions defined on $\Gamma$ with values in $M$. For any $f,g\in \OGM$ and $\alpha\in \OGR$, we define:
$$ (f+g)(n):=f(n)+g(n),\;(\forall)n\in \Gamma\;\text{ and }$$
$$ (\alpha \cdot f)(n):=\sum_{ab=n}\alpha(a)f(b),\;(\forall)n\in \Gamma.$$
We denote $\RMod$ the category of $R$-modules.

\begin{prop}
 With the above notations
\begin{enumerate}
 \item[(1)] $\OGM$ has a structure of $\OGR$-module.
 \item[(2)] $\mathcal F(\Gamma,-):\RMod \rightarrow \RMod$ is an exact covariant functor.
\end{enumerate}
\end{prop}

\begin{proof}
$(1)$ It is easy to see that $(\OGM,+)$ is a commutative group. Let $\alpha\in \OGR$, $f,g\in \OGM$ and $n\in \Gamma$. Then
$$ ((\alpha\cdot f)\cdot g)(n) = \sum_{ab=n}(\alpha\cdot f)(a)g(b) = \sum_{(a'a'')b=n}(\alpha(a')f(a''))g(b).$$
On the other hand,
$$ (\alpha\cdot (f\cdot g))(n) = \sum_{ab=n}\alpha(a)(f\cdot g)(b) = \sum_{a(b'b'')=n}\alpha(a)(f(b')g(b'')),$$
hence $(\alpha\cdot f)\cdot g = \alpha\cdot (f\cdot g)$. The other axioms of modules can be proved similarly.

$(2)$ Let $\varphi:M\rightarrow N$ be a morphism of $R$-modules. Then 
$$\mathcal F(\Gamma,\varphi)=:\varphi_*:\OGM \rightarrow \OGN, \text{ is defined by } \varphi_*(f):=\varphi\circ f, (\forall)f\in \OGM.$$
Also, if $0\rightarrow U \rightarrow M\rightarrow N \rightarrow 0$ is an exact sequence of $R$-modules, then, by straigtforward computations, we can check that
$0\rightarrow \OGU \rightarrow \OGM \rightarrow \OGN \rightarrow 0$ is exact.
\end{proof}

\begin{cor}
If $R'$ is a $R$-algebra with the structure morphism $i:R\rightarrow R'$, 
then $\mathcal F(\Gamma,R')$ has a natural structure of an $\mathcal F(\Gamma,R)$-algebra, given by 
$$i_*:\OGR\rightarrow \mathcal F(\Gamma,R'),\; i_*(f):=i\circ f$$
\end{cor}

\begin{proof}
 It follows immediately from Proposition $1.3$.
\end{proof}

\begin{prop}
Let $\Gamma' \subset \Gamma$ be a submonoid. Then:
\begin{enumerate}
 \item[(1)] The map $i_*:\mathcal F(\Gamma',R) \rightarrow \mathcal F(\Gamma,R)$, 
       $i_*(\alpha)(n):=\begin{cases} \alpha(n), & n \in \Gamma' \\ 0, & n \notin \Gamma'\end{cases}$,
       is an injective morphism of $R$-algebras, hence $\mathcal F(\Gamma,R)$ has a structure of $\mathcal F(\Gamma',R)$-algebra.
 \item[(2)] The map $i_*:\mathcal F(\Gamma',M)\rightarrow \mathcal F(\Gamma,M)$, 
       $i_*(\alpha)(n):=\begin{cases} \alpha(n), & n \in \Gamma' \\ 0, & n \notin \Gamma' \end{cases}$,
       is an injective morphism of $\mathcal F(\Gamma',R)$-modules.
\end{enumerate}
\end{prop}

\begin{proof}
$(1)$ Let $\alpha,\beta\in\mathcal F(\Gamma',R)$. Then 
$$i_*(\alpha+\beta)(n) = \begin{cases} (\alpha+\beta)(n), & n \in \Gamma' \\ 0, & n \notin \Gamma'\end{cases} = i_*(\alpha)(n)+i_*(\beta)(n), (\forall)n\in\Gamma.$$ 
Let $n\in \Gamma'$. Then
$$i_*(\alpha\cdot\beta)(n) = (\alpha\cdot\beta)(n) = \sum_{ab=n,\;a,b\in\Gamma'}\alpha(a)\beta(b) = 
\sum_{ab=n,\;a,b\in\Gamma'}i_*(\alpha)(a)i_*(\beta)(b) =  $$
$$ = \sum_{ab=n,\;a,b\in\Gamma}i_*(\alpha)(a)i_*(\beta)(b) =  (i_*(\alpha)\cdot i_*(\beta))(n)$$
Assume $n\in \Gamma\setminus\Gamma'$. It follows that if $n=ab$ with $a,b\in \Gamma$ then $a\notin\Gamma'$ or $b\notin\Gamma'$.
Therefore $$(i_*(\alpha)\cdot i_*(\beta))(n) = \sum_{ab=n}i_*(\alpha)(a)i_*(\beta)(b) = 0 = i_*(\alpha\cdot \beta)(n).$$
Also, for $r\in R$ and $n\in\Gamma$, it is clear that $i_*(r\cdot \alpha)(n)=ri_*(\alpha)(n)$.

$(2)$ The proof is similar to the proof or $(1)$.
\end{proof}

A function $L \in\mathcal F(\Gamma,R)$ is called \emph{totally multiplicative}, if it is a monoid morphism between $(\Gamma,\cdot)$ and $(R,\cdot)$, i.e. 
$$L(ab)=L(a)L(b),(\forall)a,b\in\Gamma \text{ and } L(1)=1.$$
We denote $\mathcal T(\Gamma,R)$ the set of totally multiplicative functions. 

\begin{prop}
Let $L\in \mathcal T(\Gamma,R)$ and let $M$ be an $R$-module. We have that:
\begin{enumerate}
\item[(1)] $\Phi_{L,R}:\mathcal F(\Gamma,R) \rightarrow \mathcal F(\Gamma,R),\;\Phi_{L,R}(\alpha)(n):=L(n)\alpha(n)$,
      is a morphism of $R$-algebras. 
\item[(2)] $\mathcal F(\Gamma,M)$ has a structure of a $\mathcal F(\Gamma,R)$-module given by
           $\alpha\cdot_{L} f := \Phi_{L,R}(\alpha)\cdot f$. We denote this structure by $\mathcal F_{L}(\Gamma,M)$.
\item[(3)] $\Phi_{L,M}:\mathcal F(\Gamma,M) \rightarrow \mathcal F_{L}(\Gamma,M),\;\Phi_{L,M}(f)(n):=L(n)f(n)$
           is a morphism of $\mathcal F(\Gamma,R)$-modules.
\end{enumerate}
\end{prop}

\begin{proof}
$(1)$ Let $\alpha,\beta \in \mathcal F(\Gamma,R)$. Then, for any $n\in \Gamma$, we have
$$\Phi_{L,R}(\alpha+\beta)(n) = L(n)(\alpha+\beta)(n) = L(n)\alpha(n) + L(n)\beta(n) = \Phi_{L,R}(\alpha)(n) + \Phi_{L,R}(\beta)(n),$$
hence $\Phi_{L,M}(\alpha+\beta)=\Phi_{L,M}(\alpha)+\Phi_{L,M}(\beta)$. Also,
$$\Phi_{L,R}(\alpha\cdot\beta)(n) = L(n)(\alpha\cdot\beta)(n) = L(n)\sum_{ab=n}\alpha(a)\beta(b) = 
\sum_{ab=n} L(a)\alpha(a) L(b)\beta(b) = $$
$$ = \sum_{ab=n} \Phi_{L,R}(\alpha)(a) \Phi_{L,R}(\beta)(b) 
= (\Phi_{L,R}(\alpha)\cdot \Phi_{L,R}(\beta))(n),$$
hence $\Phi_{L,M}(\alpha\cdot \beta)=\Phi_{L,M}(\alpha)\cdot \Phi_{L,M}(\beta)$. Moreover, if $r\in R$, then
$$ \Phi_{L,R}(r\cdot \alpha)(n) = L(n)r\alpha(n) = r(L(n)\alpha(n)) = r\Phi_{L,R}(\alpha)(n).$$
Finally, if $e\in\mathcal F(\Gamma,R)$ then $\Phi_{L,R}(e)(1)=L(1)e(1) = 1\cdot 1 = 1$ and $\Phi_{L,R}(e)(n) = L(n)e(n) =
L(n)\cdot 0 = 0$, hence $\Phi_{L,R}(e)=e$.

$(2)$ Let $\alpha,\beta \in \mathcal F_{L}(\Gamma,R)$ and $f,g\in\mathcal F(\Gamma,M)$. Then, for any $n\in \Gamma$, we have:
$$ \alpha\cdot_{L} (f+g) = \Phi_{L,R}(\alpha)(f+g) =  \Phi_{L,R}(\alpha)\cdot f + \Phi_{L,R}(\alpha)\cdot g = \alpha\cdot_{L} f + \alpha \cdot_{L} g,$$
$$ (\alpha+\beta)\cdot_{L} f = \Phi_{L,R}(\alpha+\beta)\cdot f = \Phi_{L,R}(\alpha)\cdot f +
\Phi_{L,R}(\beta)\cdot f =  \alpha\cdot_{L} f +  \beta\cdot_{L} f,$$
$$ \alpha\cdot_{L} (\beta \cdot_{L} f) = \Phi_{L,R}(\alpha)\cdot (\Phi_{L,R}(\beta)\cdot f) 
 = (\Phi_{L,R}(\alpha)\cdot \Phi_{L,R}(\beta))\cdot f =  (\Phi_{L,R}(\alpha \cdot \beta))\cdot f 
= (\alpha \cdot \beta)\cdot_{L} f,$$
$$ e\cdot_{L} f =  \Phi_{L,R}(e)\cdot f = e \cdot f = f.$$
$(3)$ Let $f,g \in \mathcal F(\Gamma,M)$. Then, for any $n\in \Gamma$, we have
$$\Phi_{L,M}(f+g)(n) = L(n)(f+g)(n) = L(n)f(n) + L(n)g(n) = \Phi_{L,M}(f)(n) + \Phi_{L,M}(g)(n),$$
hence $\Phi_{L,M}(f+g)=\Phi_{L,M}(f)+\Phi_{L,M}(g)$. Let $\alpha\in \mathcal F(\Gamma,R)$ and $n\in \Gamma$. Then:
$$\Phi_{L,M}(\alpha\cdot f)(n) = L(n)(\alpha\cdot f)(n) = \sum_{ab=n}L(n)\alpha(a)f(b) = 
   \sum_{ab=n}L(a)L(b)\alpha(a)f(b) =  $$
$$ = \sum_{ab=n} \Phi_{L,R}(\alpha)(a) \Phi_{L,M}(f)(b) =  (\Phi_{L,R}(\alpha)\cdot \Phi_{L,M}(f))(n)
= (\alpha \cdot_{L} \Phi_{L,M}(f))(n), $$
hence $\Phi_{L,M}(\alpha\cdot f)=\alpha_{L}\cdot \Phi_{L,M}(f)$.
\end{proof}

We recall some basic facts on derivations and K\"ahler differentials, see \cite[\S 20]{grot} for further details.
Let $A$ be a commutative ring with unity and let $i:A\rightarrow R$ be a morphism of rings with unity. 
Let $M$ be an $R$-module. An $A$-derivation $D:R \rightarrow M$ is an $A$-linear map, satisfying the Leibniz rule, i.e. 
$$D(fg)=fD(g)+gD(f), (\forall)f,g\in R.$$
The set of $A$-derivations $\Der_A(R,M)$ has a natural structure of a $R$-module. We write $\Der_A(R)$ for $\Der_A(R,R)$.

\begin{prop}
Let $D\in \Der_A(R,M)$ and let $\delta:\Gamma \rightarrow (M,+)$ be a morphism of monoids. Then
$$\widetilde D:\mathcal F(\Gamma,R) \rightarrow \mathcal F(\Gamma,M),\; \widetilde D(\alpha)(n)=D(\alpha(n))+\alpha(n)\delta(n),\;(\forall)n\in\Gamma,$$
is an $A$-derivation. Consequently, we have a natural map
$$ \Der_A(R,M)\otimes_R Hom_R(\Gamma,M) \rightarrow \Der_A(\mathcal F(\Gamma,R),\mathcal F(\Gamma,M)), D\otimes \delta \mapsto \widetilde D,$$
where $Hom_R(\Gamma,M)$ has the natural structure of a $R$-module given by $(r\cdot \delta)(n) = r\delta(n)$, for all $r\in R$, $\delta\in Hom_R(\Gamma,R)$ and $n\in\Gamma$.
\end{prop}

\begin{proof}
Let $\alpha,\beta\in \mathcal F(\Gamma,R)$. We have 
$$\widetilde D(\alpha+\beta)(n) = D((\alpha+\beta)(n)) + (\alpha+\beta)(n)\delta(n) = D(\alpha(n)+\beta(n)) + \alpha(n)\delta(n) + \beta(n)\delta(n) = $$
$$ = D(\alpha(n))+\alpha(n)\delta(n) + D(\beta(n))+\beta(n)\delta(n)  = \widetilde D (\alpha)(n)+ \widetilde D (\beta)(n),\;(\forall)n\in\Gamma.$$
Let $r\in A$ and $\alpha\in \mathcal F(\Gamma,R)$. We denote also by $r$ the image $i(r)\in R$. Then:
$$ \widetilde D(r \cdot \alpha)(n) = D((r\cdot \alpha)(n))+(r\cdot \alpha)(n)\delta(n) = D(r\cdot \alpha(n)) + r\alpha(n)\delta(n) = $$
$$ =r(D(\alpha(n))+\alpha(n)\delta(n))=r\widetilde D(\alpha)(n),(\forall)n\in \Gamma.$$
It follows that $D$ is $A$-linear. On the other hand,
$$\widetilde D(\alpha\cdot \beta)(n) = D((\alpha\cdot\beta)(n)) + (\alpha\cdot\beta)(n)\delta(n) = D(\sum_{ab=n}\alpha(a)\beta(b)) + \sum_{ab=n}\alpha(a)\beta(b)(\delta(a)+\delta(b)) = $$
$$ = \sum_{ab=n}D(\alpha(a))\beta(b) + \sum_{ab=n}\alpha(a)\beta(b)\delta(a) + 
       \sum_{ab=n}\alpha(a)D(\beta(b)) + \sum_{ab=n}\alpha(a)\beta(b)\delta(b) = $$
$$= (D(\alpha)\cdot \beta)(n) + (\alpha\cdot D(\beta))(n),\;\text{as required.}$$
\end{proof}

\section{The ring of arithmetic functions}

In this section, $(\Gamma,\cdot)$ is a submonoid in $(\mathbb N^*,\cdot)$. Let $R$ be a domain. 
We define $N:\OGR \rightarrow \mathbb N$, by setting 
$N(0)=0\;\text{ and } N(\alpha)=\min\{n\;|\;\alpha(n)\neq 0\}, \; 0 \neq \alpha\in\OGR.$

\begin{prop}(see \cite{cash})
$(\OGR,+,\cdot)$ is a domain and $N$ is a non-archimedian norm on $\OGR$, i.e.
\begin{enumerate}
\item $N(\alpha)\geq 0$, $(\forall)\alpha\in \OGR$. $N(\alpha)=0 \Longleftrightarrow \alpha=o$.
\item $N(\alpha+\beta)\leq \min\{ N(\alpha), N(\beta)\}$, $(\forall)\alpha,\beta\in \OGR$ and the equality holds if $N(\alpha)\neq N(\beta)$.
\item $N(\alpha\cdot \beta)=N(\alpha)N(\beta)$, $(\forall)\alpha,\beta\in \OGR$.
\end{enumerate}
\end{prop}

\begin{proof}
Let $\alpha,\beta\in \OGR\setminus \{0\}$ and assume that $N(\alpha)=a$
and $N(\beta)=b$. Then it is clear that $(\alpha\cdot\beta)(n)=0$ for all (if any) $n\in \Gamma$ with $n<ab$ and $(\alpha\cdot\beta)(ab)=\alpha(a)\beta(b)\neq 0$. 
Hence $N(\alpha\cdot\beta)=N(\alpha)N(\beta)$ and therefore $\alpha\cdot\beta\neq 0$.
\end{proof}

\begin{prop}(see also \cite[Theorem 6]{ber})
The group of units of $\OGR$ is
$$U(\OGR)=\{\alpha\in \OGR\;:\; \alpha(1)\in U(R) \},$$
where $U(R)$ is the group of units of $R$.
\end{prop}

\begin{proof}
Assume $\alpha$ is an unit in $\OGR$, i.e. there exists $\beta\in \OGR$ such that $\alpha\cdot\beta=e$. It follows that
$(\alpha\cdot\beta)(1)=\alpha(1)\beta(1)=e(1)=1$, hence $\alpha(1)\in U(R)$. For the converse, let $\alpha\in \OGR$ with
$\alpha(1)\in U(R)$. We define
$$\beta(1):=\alpha(1)^{-1},\; \beta(n):=-\alpha(1)^{-1} \sum_{ab=n,\; b\neq n} \alpha(a)\beta(b), (\forall)n\in \Gamma\setminus\{1\}.$$
It is easy to check that $\alpha\cdot \beta=e$.
\end{proof}

\begin{obs}
\emph{Let  $i:R \rightarrow \OGR$ defined by $i(r)(1)=r$ and $i(r)(n)=0$, $(\forall)n\geq 2$. The map $\pi:\OGR \rightarrow R$, $\pi(\alpha):=\alpha(1)$ is a ring epimorphism. 
One has  $\pi\circ i = Id_R$,  $\Ker \Phi = \{\alpha\in\OGR,\;\alpha(1)=0\}$ and $\OGR/\Ker\Phi \cong R$.
Moreover, according to Proposition $1.2$, we have $U(\OGR)=\Phi^{-1}(U(R))$. See also \cite[Remark 1]{ber}.}
\end{obs}

In the following, we tackle the case $\Gamma=\mathbb N^*$.
Let $\mathbb P=\{p_1,p_2,p_3,\ldots\}$ be the set of prime numbers. We assume that $p_1<p_2<p_3<\cdots$,
i.e. $p_1=2$, $p_2=3$, $p_3=5$ etc. For any $k\geq 1$, we let $\Gamma(k)$ be the submonoid in $\mathbb N^*$
generated by $\mathcal P_k:=\{p_1,\ldots,p_k\}$, i.e. $$\Gamma(k)=\{p_1^{a_1}\cdots p_k^{a_k}\;:\; a_i\in\mathbb N\}.$$
Note that $(\Gamma(k),\cdot)$ is isomorphic to $(\mathbb N^k,+)$ and 
$$\mathbb N^* = \lim_{\longrightarrow} \Gamma(k) = \bigcup_{k\geq 1}\Gamma(k).$$
The map $$\Phi_k:\mathcal F(\mathbb N^*,R) \rightarrow \mathcal F(\Gamma(k),R), \Phi_k(\alpha)=\alpha|_{\Gamma(k)},$$
is a morphism of $R$-algebras. Also, for any $k\leq k'$, the map 
$$\Phi_{k,k'}:\mathcal F(\Gamma(k'),R) \rightarrow \mathcal F(\Gamma(k),R), \Phi_k(\alpha)=\alpha|_{\Gamma(k)},$$
is a morphism of $R$-algebras. One can easily check that $(\mathcal F(\Gamma(k),R),\Phi_{k,k'})_{k,k'\geq 1}$ is an inverse system
and 
\begin{equation}
\mathcal F(\mathbb N^*,R) = \lim_{\longleftarrow} \mathcal \mathcal F(\Gamma(k),R).
\end{equation}

\begin{prop}
With the above notations, we have the $R$-algebra isomorphisms 
$$\mathcal F(\Gamma(k),R) \cong R[[x_1,\ldots,x_k]],(\forall)k\geq 1, \text{ and } \mathcal F(\mathbb N^*,R) \cong R[[x_1,x_2,x_3,\ldots ]].$$
\end{prop}

\begin{proof}
We define $\Phi_k:\mathcal F(\Gamma(k),R) \rightarrow R[[x_1,\ldots,x_k]]$ by 
$$\Phi_k(\alpha)=\sum_{n\in \Gamma(k)}\alpha(n)x_1^{a_{n1}}\cdots x_k^{a_{nk}},\text{ where } n=p_1^{a_{n1}}\cdots p_k^{a_{nk}},$$
for any $\alpha\in \mathcal F(k)$. It is easy to check that $\Phi_k$ is an $R$-algebra isomorphism. By $(2.1)$, it follows that
$$\mathcal F(\mathbb N^*,R) \cong \lim_{\longleftarrow} R[[x_1,\ldots,x_k]] = R[[x_1,x_2,x_3,\ldots ]].$$
\end{proof}

\begin{cor}
Let $(\Gamma,\cdot)$ be a submonoid of $(\mathbb N^*,\cdot)$.
\begin{enumerate}
 \item[(1)] $\OGR \cong R[[x_1,x_2,\ldots]]/I_{\Gamma}$, where $I_{\Gamma}\subset R[[x_1,x_2,\ldots ]]$ is a prime ideal.
 \item[(2)] If $\Gamma$ is finitely generated, then there exists $k\geq 1$ such that $\OGR\cong R[[x_1,\ldots,x_k]]/I_{\Gamma}$, where $I_{\Gamma}$ is a prime ideal.
\end{enumerate}
\end{cor}

\begin{proof}
$(1)$ Assume that $\Gamma$ is generated by $\{n_1,n_2,n_3,\ldots\}$. We define:
$$\Phi: R[[x_1,x_2,\ldots ]] \rightarrow \OGR,\; x_i\mapsto \alpha^i,\; \alpha^i(n)=\begin{cases} 1, & n = n_i \\ 0, & n\neq n_i \end{cases}.$$
If $u=x_1^{a_1}\cdots x_r^{a_r}$, then $\Phi(u)(n)=\begin{cases} 1, & n = n_1^{a_1}\cdots n_r^{a_r} \\ 0, &  \text{ otherwise } \end{cases}$.
Hence, if $f=\sum_{n=1}^{\infty}\alpha(n)u_n$, where $u_n=x_1^{a_{n1}}\cdots x_n^{a_{nr}}$ for $n=p_1^{a_{n1}}\cdots p_r^{a_{nr}}$.

We have the canonical surjective morphism $\mathcal F \rightarrow \OGR$, $\alpha\mapsto \alpha|_{\Gamma}$, hence $\OGR$ can be described as a quotient ring of
$\mathcal F\cong R[[x_1,x_2,\ldots]]$. Since $\OGR$ is a domain, it follows that $\OGR \cong R[[x_1,x_2,\ldots]]/I_{\Gamma}$, 
where $I_{\Gamma}\subset R[[x_1,x_2,\ldots ]]$ is a prime ideal.

$(2)$ Assume that $\Gamma$ is generated, as a monoid, by $n_1,\ldots,n_r$. We chose $k\geq 1$ such that all the prime factors 
of $n_1,\ldots,n_r$ are in the set $\{p_1,p_2,\ldots,p_k\}$. It follows that $\Gamma\subset \Gamma(k)$, hence $\OGR$ is isomorphic to
a quotient ring of $\mathcal F(k)\cong R[[x_1,\ldots,x_k]]$. Since $\mathcal F(k)$ is a domain, it follows that $I_{\Gamma}$ is a prime ideal.
\end{proof}

\begin{obs}\emph{
 Let $(\Gamma,+) \subset (\mathbb N^k,+)$, $k\geq 1$, be a finitely generated submonoid, i.e. $\Gamma$ is an \emph{affine positive monoid}.
 Assume that $\Gamma$ is generated by the vectors $v_1,\ldots,v_r \in \mathbb N^k$. Chosing $p_1,\ldots,p_k$ a set of distinct prime numbers, 
 the map $i:\Gamma \rightarrow \mathbb N^*$, $i(a_1,\ldots,a_k)= p_1^{a_1}\cdots p_r^{a_k}$, is an injective morphism of monoids, hence we can
 see $\Gamma$ as a submonoid in $(\mathbb N^*,\cdot)$.
 Then, according to Corollary $2.5(2)$, the map 
 $$\Phi_{\Gamma}:R[[x_1,\ldots,x_k]] \rightarrow \mathcal F(\Gamma,R),\; \Phi_{\Gamma}(\sum_{(a_1,\ldots,a_k)\in \mathbb N^k}c_{a_1,\ldots,a_k}x_1^{a_1}\cdots x_k^{a_k})(a_1,\ldots,a_k):= c_{a_1,\ldots,a_k},$$
 is a surjective morphism of $R$-algebras, hence $\mathcal F(\Gamma,R)\cong R[[x_1,\ldots,x_k]]/\Ker(\Phi)$. The ideal $I_{\Gamma}:=\Ker(\Phi_{\Gamma})$ is prime and it is called
 the} toric ideal \emph{of $\Gamma$.}
\end{obs}

An \emph{unique factorization domain} (UFD) is a domain $R$ in which every non-zero element $x$ can be written as a product 
$x=u\pi_1\cdots\pi_r$, where $u\in U(R)$ is an unit and $\pi_i\in R$ are prime elements, i.e. if $\pi_i|ab$ then $\pi_i|a$ or $\pi_i|b$. 
A ring $R$ is called \emph{regular} if $R$ is Noetherian and for every prime ideal $P\subset R$, 
the localization $R_P$ is a local regular ring. 

Note that, if $R$ is UFD then, in general, the ring of power series $R[[x]]$ is not, see
\cite[Proposition 14]{samuel}. However, if $R$ is regular and UFD, then $R[[x_1,\ldots,x_r]]$ is also regular and UFD, for any $r\geq 1$, 
see \cite[Theorem 2.1]{sam} or \cite[Theorem 3.2]{buchsbaum}. We recall the following result.

\begin{teor}(Cashwell-Everret \cite{cash2})
If $R$ is an UFD with the property that 
$$R[[x_1,\ldots,x_r]] \text{ is UFD }, (\forall)r\geq 1,$$ then $\mathcal F(\mathbb N^*,R)$ is UFD.
In particular, if $R$ is regular UFD, then $\mathcal F(\mathbb N^*,R)$ is UFD.
\end{teor}

Let $p\in\mathbb P$ be a prime and $n\in\mathbb N^*$. We denote $v_p(n)$ the exponent of the highest power of $p$ dividing $m$.
According to \cite[Lemma 4.5.1]{shapiro}, the map $$D_p:\mathcal F(\mathbb N^*,\mathbb C) \rightarrow \mathcal F(\mathbb N^*,\mathbb C),\; D_p(\alpha)(n):=v_p(np)\alpha(np),$$
is a $\mathbb C$-derivation on $\mathcal F(\mathbb N^*,\mathbb C)$. Also, the map $$D_L:\mathcal F(\mathbb N^*,\mathbb C) \rightarrow \mathcal F(\mathbb N^*,\mathbb C),\;D_L(\alpha)(n)= -\log (n) \alpha(n),$$
is a $\mathbb C$-derivation which is called the \emph{log-derivation}. The following propositions extend these observations.

\begin{prop}
Let $A$ be a domain and let $R$ be an $A$-algebra. Let $D:R\rightarrow R$ be an $A$-derivation. Let $p\in\mathbb P$. The map
$$\widetilde D:\mathcal F(\mathbb N^*,R)\rightarrow \mathcal F(\mathbb N^*,R),\; \widetilde D_p(\alpha)(n)= D(\alpha(n)) + \alpha(np)v_p(np),$$
is an $A$-derivation of $\mathcal F(\mathbb N^*,R)$.
\end{prop}

\begin{proof}
 Let $\alpha,\beta\in F(\mathbb N^*,R)$. We have that
$$ \widetilde D_p(\alpha+\beta)(n) = D((\alpha+\beta)(n)) + (\alpha+\beta)(np)v_p(np) = D(\alpha(n)+\beta(n)) + (\alpha(np)+\beta(np))v_p(np) = $$
$$ = D(\alpha(n)) + \alpha(np)v_p(np)+ D(\beta(n)) + \beta(np)v_p(np) = \widetilde D_p(\alpha)(n) + \widetilde D_p(\beta)(n).$$
Let $r\in A$ and $\alpha\in F(\mathbb N^*,R)$. As in the proof of Proposition $1.7$, we identify $r$ with $i(r)\in R$, where $i:A\rightarrow R$ is the
structure morphism of the $A$-algebra $R$. We have that
$$ \widetilde D_p(r\alpha)(n) = D((r\alpha)(n)) + (r\alpha)(np)v_p(np) = D(r\alpha(n)) + r\alpha(np)v_p(np) = $$ $$ = rD(\alpha(n))+ r\alpha(np)v_p(np) = r\widetilde D_p(r\alpha)(n).$$
It follows that $\widetilde D_p$ is $A$-linear. Let $\alpha,\beta\in F(\mathbb N^*,R)$. We have that
$$\widetilde D_p(\alpha\cdot \beta)(n) = D((\alpha\cdot \beta)(n)) + (\alpha\cdot \beta)(np)v_p(np) = D(\sum_{ab=n}\alpha(a)\beta(b)) + \left(\sum_{ab=np}\alpha(a)\beta(b)\right)v_p(np) = $$
\begin{equation}\label{der1}
 =  \sum_{ab=n}D(\alpha(a))\beta(b) + \sum_{ab=n}\alpha(a)D(\beta(b)) + \sum_{ab=np} \alpha(a)\beta(b)(v_p(a)+v_p(b)).
\end{equation}
On the other hand,
$$(\widetilde D_p(\alpha)\cdot \beta)(n) =  \sum_{ab=n}\widetilde D_p(\alpha)(a)\beta(b) = \sum_{ab=n} (D(\alpha(a)) + \alpha(ap)v_p(ap))\beta(b) = $$
\begin{equation}\label{der2}
 \sum_{ab=n} (D(\alpha(a))\beta(b) +  \sum_{ab=n} \alpha(ap)v_p(ap)\beta(b).
\end{equation}
On the other hand, as $v_p(1)=0$, it follows that 
\begin{equation}\label{der3}
\sum_{ab=n}\alpha(ap)v_p(ap)\beta(b) = \sum_{ab=np}\alpha(a)v_p(a)\beta(b) 
\end{equation}
From (\ref{der1}), (\ref{der2}) and (\ref{der3}) it follows that
$$\widetilde D_p(\alpha\cdot \beta)  = \widetilde D_p(\alpha)\cdot \beta + \alpha \cdot \widetilde D_p(\beta),$$
as required.
\end{proof}

\begin{prop}
Let $U\subset \mathbb R$ ($U\subset \mathbb C$) be an open set and let $\mathcal O(U)$ be the ring of $\mathbb R$-derivable (holomorphic) functions defined in $U$ with values in $\mathbb R$ 
(respectively $\mathbb C$). We consider 
$$\widetilde D: \mathcal F(\mathbb N^*,\mathcal O(U)) \rightarrow \mathcal F(\mathbb N^*,\mathcal O(U)), \; \widetilde  D(\alpha)(n)(z):=\alpha(n)'(z) - \alpha(n)(z) \log n, \; (\forall)n\in \mathbb N, z\in U.$$
Then $\widetilde D$ is a $\mathbb R$-derivation ($\mathbb C$-derivation) on $\mathcal F(\mathbb N^*,\mathcal O(U))$.
In the language of the sheaves theory, $\mathcal F(\mathbb N^*,\mathcal O)$ is a sheaf of $\mathcal O$-algebras and $\widetilde D:\mathcal F(\mathbb N^*,\mathcal O) \rightarrow \mathcal F(\mathbb N^*,\mathcal O)$ 
is a $\mathbb R$-derivation (respectively $\mathbb C$-derivation).
\end{prop}

\begin{proof}
For any $n\in\mathbb N^*$, we define $\delta(n):=-\log n$. 
It is easy to check that the hypothesis of Proposition $1.7$ are satisfied, 
hence $\widetilde D$ is a $\mathbb R$-derivation ($\mathbb C$-derivation) on $\mathcal F(\mathbb N^*,\mathcal O(U))$.
\end{proof}

\begin{obs}
\emph{Let $U\subset \mathbb R$ ($U\subset \mathbb C$) be an open set.
Let $\alpha\in \mathcal F(\mathbb N^*,\mathcal O(U))$ such that the series of functions $F_{\alpha}(x):=\sum_{n=1}^{\infty} \frac{\alpha(n)(x)}{n^x}$ 
and $G_{\alpha}(x)= \sum_{n=1}^{\infty} \left(\frac{\alpha(n)(x)}{n^x}\right)'$ are uniformly convergent on the compact subsets $K\subset U$. It is well known that, in this case, $F$ defines a derivable
(holomorphic) function on $U$ and, moreover, $F'=G$. In other words, $F'_{\alpha} = F_{\widetilde D(\alpha)}$, where $\widetilde D$ is defined in Proposition $2.9$.}

\emph{Let $L \in \mathcal T(\mathbb N^*,\mathcal O(U))$, see the notations from section $1$. We define the map
$$\Phi_{L}: \mathcal F(\mathbb N^*,\mathcal O(U)) \rightarrow \mathcal F(\mathbb N^*,\mathcal O(U)),\; \Phi_{L}(\alpha)(n)(z):=L(n)(z)\alpha(n)(z),\;(\forall)n\geq 1,\; z\in U.$$
According to Proposition $1.6$, $\Phi_{L}$ is a $\mathcal O(U)$-algebra endomorphism.}
\end{obs}

\section{An extension of the ring of Dirichlet convolutions}

Let $\Gamma$ be a commutative monoid of finite type. We assume that $(\Gamma,\cdot)$ is \emph{cancellative}, i.e. $xy = xz$ implies $y=z$. 
We recall the construction of the \emph{Grothendieck group}, denoted by $G(\Gamma)$, associated to $\Gamma$, see \cite{atiyah} for further details. 
On the set $\Gamma\times\Gamma$ we define the relation $$ (m,n)\sim (p,q) \Leftrightarrow mq=np,\;m,n,p,q\in \Gamma.$$
One can easily check that $\sim$ is a relation of equivalence on $\Gamma$ which is compatible with the multiplication.
We denote $G(\Gamma):=\Gamma\times\Gamma/\sim$. We denote $\frac{m}{n}$ the class of $(m,n)$ in $G(\Gamma)$.
We define the operation $\cdot$ on $G(\Gamma)$ by $$ \frac{m}{n}\cdot\frac{m'}{n'}:=\frac{mm'}{nn'},\;(\forall)m,n,m',n'\in\Gamma.$$
The identity in $G(\Gamma)$ is $1=\frac{1}{1}$. If $q=\frac{m}{n}\in G(\Gamma)$, then $q^{-1}=\frac{n}{m}$.
For example, it holds that $G((\mathbb N,+))=(\mathbb Z,+)$ and $G((\mathbb N^*,\cdot))=(\mathbb Q_+^*,\cdot)$.

\begin{prop}(universal property)
The map $i:\Gamma \rightarrow G(\Gamma)$, $n\mapsto \frac{n}{1}$, is an injective morphism of monoids, and 
for any commutative group $(G,\cdot)$ and morphism of monoids $f:\Gamma\rightarrow G$, there exists a morphism of groups
 $\bar f:G(\Gamma)\rightarrow G$ such that $\bar f \circ i = f$. 
\end{prop}

\begin{proof}
The map $\bar f$ is defined by $$\bar f(\frac{m}{n}):= f(m)\cdot f(n)^{-1},\; (\forall)m,n\in\Gamma.$$
The conclusion follows by straightforward computations.
\end{proof}

Let $\Gamma$ be a commutative monoid of finite type, cancellative. By Proposition $3.1$, we can see $\Gamma$ as a subset of $G(\Gamma)$.
Let $R$ be a commutative ring with unity. In the set
$$\wF(G(\Gamma),R):=\{\alpha:G(\Gamma) \rightarrow R\;:\; (\exists)d\in\Gamma \text{ such that } (\forall)q\in G(\Gamma), 
\alpha(q)\neq 0 \Rightarrow dq\in\Gamma\}.$$
we consider the operations
$$ (\alpha+\beta)(q):=\alpha(q)+\beta(q),\;(\forall)q\in G(\Gamma), $$
$$ (\alpha \cdot \beta)(q):=\sum_{q'q''=q}\alpha(q')\beta(q''),\;(\forall)q\in G(\Gamma).$$
We prove the following result.

\begin{prop}
$(\wF(G(\Gamma),R),+,\cdot)$ is an $R$-algebra. Moreover, we have a canonical injective morphism of $R$-algebras
$$i_*:\mathcal F(\Gamma,R) \rightarrow \wF (G(\Gamma),R),\; i_*(\alpha)(q):=\begin{cases}
  \alpha(q),\;q\in \Gamma \\ 0,\;q\notin \Gamma \end{cases},$$
hence $\wF(G(\Gamma),R)$ is an extension of the ring $\mathcal F(\Gamma,R)$.
\end{prop}

\begin{proof}
First, we prove that $+$ and $\cdot$ are well defined. Let $d',d''\in\Gamma$ such that,
for any $q\in G(\Gamma)$, if $\alpha(q)\neq 0$, then $d'q\in \Gamma$, and if $\beta(q)\neq 0$, then $d''q\in \Gamma$.
Let $d:=d'd''$. Assume that $(\alpha+\beta)(q)\neq 0$. It follows that either $\alpha(q)\neq 0$, either $\beta(q)\neq 0$.
If $\alpha(q)\neq 0$, then $dq=(d'd'')q=d''(d'q)\in\Gamma$. If $\beta(q)\neq 0$, then $dq=(d'd'')q=d'(d''q)\in\Gamma$. 
Hence $\alpha+\beta\in \wF(G(\Gamma),R)$.

Now, let $q,q',q''\in G(\Gamma)$ such that $q'q''=q$. Since $dq=d'd''q'q''=(d'q')(d''q'')\in \Gamma$, it follows that
$$ (\alpha \cdot \beta)(q):=\sum_{ab=dq}\alpha(\frac{a}{d'})\beta(\frac{b}{d''}),$$
hence the multiplication is well defined by the fact that $\Gamma$ is of finite type.
Also, $(\alpha\cdot \beta)(q)\neq 0$ implies the fact that there exists $q',q''\in G(\Gamma)$ such that $q=q'q''$, 
$\alpha(q')\neq 0$ and $\beta(q'')\neq 0$. As above, we have that $dq\in \Gamma$, hence $\alpha\cdot \beta \in \wF(G(\Gamma),R)$.

The proof of the fact that $(\wF(G(\Gamma),R),+,\cdot)$ is an $R$-algebra is similar to the proof of Proposition $1.1$, 
hence we skip it. The proof of the fact that $i_*$ is an injective morphism of $R$-algebras is similar to the proof of Proposition $1.5(1)$.
 \end{proof}

Let $M$ be an $R$-module. We consider the set
$$\wF(G(\Gamma),M):=\{f:G(\Gamma) \rightarrow M\;:\; (\exists)d\in\Gamma \text{ such that } (\forall)q\in G(\Gamma) 
f(q)\neq 0 \Rightarrow dq\in\Gamma\}.$$
For $f,q\in\wF(G(\Gamma),M)$ we define $(f+g)(q):=f(q)+g(q),(\forall)q\in G(M)$.
Given $\alpha\in \wF(G(\Gamma),R)$ and $f\in\wF(G(\Gamma),M)$ we define 
$$(\alpha\cdot f)(q):= \sum_{q'q''=q}\alpha(q')f(q''),\;(\forall)q\in G(\Gamma).$$

Similarly to Proposition $3.2$, we have the following result.

\begin{prop}
$\wF(G(\Gamma),M)$ has a structure of an $\wF(G(\Gamma),R)$-module. Moreover, we have a canonical injective morphism of $\wF(\Gamma,R)$-modules
$$i_*:\mathcal F(\Gamma,M) \rightarrow \wF (G(\Gamma),M),\; i_*(\alpha)(q):=\begin{cases}
  \alpha(q),\;q\in \Gamma \\ 0,\;q\notin \Gamma \end{cases}.$$
\end{prop}

\begin{proof}
 The proof is similar to the proof of Proposition $1.3$ and Proposition $1.5(2)$.
\end{proof}


\begin{prop}
Let $A$ be a commutative ring with unity and let $R$ be an $A$-algebra.
Let $D\in \Der_A(R,M)$ be an $A$-derivation and let $\bar \delta:(G(\Gamma),\cdot) \rightarrow (M,+)$ be a morphism of groups.
Then $$\overline D:\wF(G(\Gamma),R)\rightarrow \wF(G(\Gamma),M),\; \overline D(\alpha)(q):=D(\alpha(q)) + \alpha(q)\bar \delta(q),
\;(\forall)q\in G(\Gamma),$$
is an $A$-derivation.
\end{prop}

\begin{proof}
The proof if similar to the proof of Proposition $1.7$.
\end{proof}

\begin{obs}
\emph{
According to Proposition $3.1$, if $\delta:(\Gamma,\cdot)\rightarrow (M,+)$ is a morphism of monoids, then $\delta$
extends naturally to a morphism of groups 
$$\bar \delta:(G(\Gamma),\cdot)\rightarrow (M,+),\; \bar \delta(\frac{m}{n}):=\delta(m)-\delta(n),\;(\forall)m,n\in\Gamma.$$
It follows that the derivation $\widetilde D:\mathcal(\Gamma,R)\rightarrow \mathcal(\Gamma,M)$ defined in Proposition $1.7$
extends naturally the derivation $\overline D:\wF(G(\Gamma),R)\rightarrow \wF(G(\Gamma),M)$ defined in Proposition $3.3$.}

\emph{Moreover, the following diagramm is commutative:
$$ \mathcal F(\Gamma,R)\stackrel{\widetilde D}{\longrightarrow} \mathcal F(\Gamma,M)$$
$$i_* \downarrow \hspace{70pt}  \downarrow i_*$$
$$ \wF(\Gamma,R)\stackrel{\overline D}{\longrightarrow} \wF(\Gamma,M).$$}
\end{obs}

In the following, we specialize $\Gamma$ to $(\mathbb N^*,\cdot)$, hence $G(\Gamma)=G(\mathbb N^*)=\mathbb Q^*_+$.
Similary to Proposition $2.4$ we show the following result.

\begin{prop}
We have that the $R$-algebra isomorphism
$$\mathcal F(\mathbb Q_+^*,R)\cong R[[x_1,x_2,\ldots]][x_1^{-1},x_2^{-1},\ldots].$$
\end{prop}

\begin{proof}
Let $\alpha\in \mathcal F(\mathbb Q_+^*,R)$. As in the section $2$, let $\mathbb P=\{p_1,p_2,p_3,\ldots\}$ be the set of prime numbers.
Any positive rational number $q\in\mathbb Q_+^*$ can be written as $q=p_1^{a_1}\cdots p_r^{a_r}$, where $r\geq 1$ and $a_i\in \mathbb Z$. 
We denote $q^x:=x_1^{a_1}\cdots x_r^{a_r}$ and
we define $$\Phi:\mathcal F(\mathbb Q_+^*,R) \rightarrow R[[x_1,x_2,\ldots]][x_1^{-1},x_2^{-1},\ldots], \; \Phi(\alpha):=\sum_{q\in\mathbb Q}\alpha(q)q^x.$$
For any $\alpha \in \mathcal F(\mathbb Q_+^*,R)$ there exist a positive integer $d$ such that if $\alpha(q)\neq 0$ for some 
$q\in \mathbb Q_{+}^*$ it follows that $dq\in\mathbb N^*$. We may assume that $d=p_1^{b_1}\cdots p_r^{b_r}$, where 
$r,b_1,\ldots,b_r$ are some nonnegative integers. In follows that $\sum_{q\in\mathbb Q}\alpha(q)(qd)^x \in R[[x_1,x_2,\ldots,]]$, hence $\Phi(\alpha)$ is well
defined as an element in the ring $R[[x_1,x_2,\ldots]][x_1^{-1},x_2^{-1},\ldots]$.
It is easy to check that $\Phi$ is an isomorphism of $R$-algebras.
\end{proof}

\begin{cor}
If $R$ is an UFD with the property that 
$$R[[x_1,\ldots,x_r]] \text{ is UFD }, (\forall)r\geq 1,$$ then $\wF(\mathbb N^*,R)$ is UFD.
\end{cor}

\begin{proof}
According to Theorem $2.7$, the ring $R[[x_1,x_2,\ldots,]]$ is UFD. On the other hand,
$$R[[x_1,x_2,\ldots]][x_1^{-1},x_2^{-1},\ldots] \cong S^{-1}R[[x_1,x_2,\ldots]],$$
where $S$ is the multiplicatively closed set consisting in all the monomials from $R[[x_1,x_2,\ldots]]$.
Since the localization of an UFD is an UFD, the conclusion follows from Proposition $3.6$.
\end{proof}

\begin{obs}\emph{
If $G$ is an ordered abelian group and $K$ is a field, the field of \emph{Hahn series} $K((T^{G}))$ with the value group $G$ is
the set $$K((T^{G}))=\{\alpha:G \rightarrow K\;:\; \{g\in G:\; \alpha(g)\neq 0\} \text{ is well ordered } \},$$
with the operations
$(\alpha+\beta)(g)=\alpha(g)+\beta(g),\; (\alpha\cdot\beta)(g)=\sum_{g'g''=g}\alpha(g')\beta(g'')$,
where the last sum is taken for $\alpha(g'),\beta(g'')\neq 0$. See \cite{hahn} and \cite{neum} for further details.
If $G=\mathbb Q_{+}^*=G(\mathbb N^*)$, then we have that $\wF(G,K)\subset K((T^G))$.}
\end{obs}

{}

\vspace{2mm} \noindent {\footnotesize
\begin{minipage}[b]{15cm}
Mircea Cimpoea\c s, Simion Stoilow Institute of Mathematics, Research unit 5, P.O.Box 1-764,\\
Bucharest 014700, Romania, E-mail: mircea.cimpoeas@imar.ro
\end{minipage}}

\end{document}